\documentclass[12pt,a4paper]{article}
\usepackage[top=2.5cm,bottom=2.5cm,left=2.5cm,right=2.5cm]{geometry}

\usepackage{amssymb,amsmath,graphicx,color,amsthm}

\newtheorem{claim}{Claim}
\newtheorem{theorem}{Theorem}
\newtheorem{lemma}[theorem]{Lemma}
\newtheorem{corollary}[theorem]{Corollary}
\newtheorem{proposition}[theorem]{Proposition}

\newtheorem{conjecture}[theorem]{Conjecture}

\usepackage[active]{srcltx}

\newenvironment{proofclaim}[1][]%
    {\noindent \emph{Proof.} {}{#1}{}}{$~$\hfill $~\blacklozenge$ \vspace{0.2cm}}

\newcommand{\set}[1]{\ensuremath{\left\{#1 \right\}}}

\newcommand{\Sz}{\hbox{\rm Sz}\,}

\begin{document}

\title{On the difference between the Szeged \\and Wiener index}

\author
{
  Marthe Bonamy\thanks{University of Waterloo, Canada. 
    E-Mail: \texttt{mbonamy@uwaterloo.ca}}, \quad
  Martin Knor\thanks{Slovak University of Technology in Bratislava, 
    Faculty of Civil Engineering, Department of Mathematics, Bratislava, Slovakia. 
    E-Mail: \texttt{knor@math.sk}}, \quad
  Borut Lu\v{z}ar\thanks{Faculty of Information Studies, Novo mesto, Slovenia. E-Mail: \texttt{borut.luzar@gmail.com}},\\
  Alexandre Pinlou\thanks{Universit\'e Montpellier 3, CNRS, LIRMM, Montpellier, France. 
    E-Mail: \texttt{alexandre.pinlou@lirmm.fr}},\quad
  Riste \v{S}krekovski\thanks{Faculty of Information Studies, Novo mesto \& FMF, University of Ljubljana \&
    FAMNIT, University of Primorska, Koper, Slovenia. E-Mail: \texttt{skrekovski@gmail.com}}
}

\maketitle

{
  \begin{abstract}
    \noindent
    We prove a conjecture of Nadjafi-Arani, Khodashenas and Ashrafi on the difference between the Szeged and Wiener index of a graph.
    Namely, if $G$ is a 2-connected non-complete graph on $n$ vertices, then $\Sz(G)-W(G)\ge 2n-6$.
    Furthermore, the equality is obtained if and only if $G$ is the complete graph $K_{n-1}$ with an extra vertex attached to
    either $2$ or $n-2$ vertices of $K_{n-1}$.
    We apply our method to strengthen some known results on the difference between 
    the Szeged and Wiener index of bipartite graphs, graphs of girth at least five, and the difference between 
    the revised Szeged and Wiener index. We also propose a stronger version of the aforementioned conjecture.
  \end{abstract}
}

\bigskip
{\noindent\small \textbf{Keywords:} Wiener index, Szeged index, revised Szeged index, Szeged--Wiener relation}


\section{Introduction}

Graph theoretic invariants of molecular graphs, which predict properties
of the corresponding molecules, are known as topological indices or molecular descriptors.
The oldest and most studied topological index is the Wiener index introduced
in 1947 by Wiener~\cite{Wie47}, who observed that this invariant can be used for
predicting the boiling points of paraf\mbox f\mbox ins.
For a simple graph $G=(V,E)$, the Wiener index is defined as 
$$
W(G)=\sum_{\{a, b\} \subseteq V} d(a,b),
$$ 
i.e. the sum of distances between all pairs of vertices. After 1947, 
the same quantity has been studied and referred to by mathematicians as the gross status \cite{Har59}, the distance of graphs \cite{EntJacSny76}, and the transmission \cite{Sol91}.
A great deal of knowledge on Wiener index is accumulated in several survey papers, see e.g. \cite{KnoSkr14,XuLiuDasGutFur14} for more recent ones.

Up to now, over 200 topological indices were introduced as potential molecular descriptors.
The definition of Szeged index in~\cite{Gut94,KhaDesKal95} was motivated
by the original definition of Wiener index for trees.
It is defined as
$$
\Sz(G)= \sum_{ab \in E} n_{ab}(a)\cdot n_{ab}(b),
$$ 
where $n_{ab}(a)$ is the number of vertices strictly closer to $a$ than $b$, and analogously, $n_{ab}(b)$ is the number of vertices strictly closer to $b$. 
Note that $n_{ab}(a)$ and $n_{ab}(b)$ are always positive. 

In this paper we consider possible values of the difference
$$
\eta(G) = \Sz(G) - W(G)
$$
between the Szeged and the Wiener index of a graph $G$. Klav\v zar et al.~\cite{KlaRajGut96} proved that the
inequality $\eta(G) \ge 0$ holds for every connected graph $G$. Moreover, Dobrynin and Gutman~\cite{DobGut95}
showed that the equality is achieved if and only if $G$ is a \textit{block graph}, i.e. a graph in which every block (maximal $2$-connected subgraph) induces a clique.
Nadjafi-Arani et al.~\cite{NadKhoAsh11,NadKhoAsh12} further investigated
the properties of $\eta(G)$ and proved that for every positive integer $k$,
with 
$k \notin \{1,3\}$, there exists a graph $G$ with $\eta(G) = k$.
Additionally, they classified the graphs $G$ for which
$\eta(G) \in \{2,4,5\}$ and asked about a general classification;
namely, can we characterize all graphs with a given value of $\eta(G)$?  They
proposed the following conjecture.

\begin{conjecture}[Nadjafi-Arani et al., 2012]
 \label{conj:main}
 Let $G$ be a connected graph and let $B_1,\ldots,B_k$ be all its non-complete
 blocks of respective orders $n_1,\ldots,n_k$.
 Then
 $$
 	\eta(G)\ge \sum_{i=1}^k (2n_i-6).
 $$
\end{conjecture}

In this paper we prove the following statement which deals with 2-connected
graphs.

\begin{theorem}
 \label{th:main1}
 If $G$ is a $2$-connected non-complete graph on $n$ vertices, then 
 $$
 	\eta(G) \geq 2n-6.
 $$
\end{theorem}

As a consequence of Theorem~\ref{th:main1} we obtain that Conjecture~\ref{conj:main}
is true.

\begin{corollary}
 \label{th:main2}
 Let $B_1,\ldots,B_k$ be all the non-complete blocks of $G$ with respective
 orders $n_1,\ldots,n_k$.
 Then 
 $$
 	\eta(G)\ge \sum_{i=1}^k (2n_i-6).
 $$
\end{corollary}

In fact, we also characterize the graphs achieving equality in Theorem~{\ref{th:main1}}.
For $n, t \in \mathbb{N}$ with $1\le t\leq n-1$, let $K^t_n$ be the graph
obtained from $K_{n-1}$ by adding one new vertex adjacent to $t$ of the $n-1$ old vertices.
Observe that $K^t_n$ is $2$-connected and non-complete if $2\le t\le n-2$.
We prove the following stronger version of Theorem~\ref{th:main1}.

\begin{theorem}
\label{th:main3}
If $G$ is a $2$-connected non-complete graph on $n$ vertices that is
not isomorphic to $K^2_n$ or $K^{n-2}_n$, then 
$$
	\eta(G) \geq 2n-5.
$$
\end{theorem}

While $\eta(K_n) = 0$, and $\eta(K_n^{2}) = \eta(K_n^{n-2}) = 2n-6$ (see Lemma~\ref{lem:eta(KK)} for the proof),
there seems to be only a finite number of graphs $G$ with $\eta(G) < 2n$; in particular, using a computer, 
we found such graphs of order at most $9$, but none on $10$ vertices. We therefore propose
the following conjecture.

\begin{conjecture}
	Let $G$ be a $2$-connected graph of order $n \ge 10$ not isomorphic to $K_n$, $K_n^{2}$, or $K_n^{n-2}$. 
	Then 
	$$
		\eta(G) \ge 2n.
	$$
\end{conjecture}

In Section~{\ref{sec:prel}} we derive Corollary~\ref{th:main2} from Theorem~\ref{th:main1}. 
We present the proofs of Theorems~\ref{th:main1} and~\ref{th:main3}
in Section~\ref{sec:main}, after introducing four technical lemmas in Section~\ref{sec:lemmas}. 
In Section~\ref{sec:add}, we apply our method in order to obtain stronger versions of other results related 
to the difference between the Szeged and Wiener index of a graph and, as corollaries, we present alternative proofs of the existing results.
Finally, in Section~\ref{sec:revised} we use our approach to prove results for the revised Szeged index.


\section{Preliminaries}
\label{sec:prel}

In the paper we will use the following definitions and notation.
The \emph{distance} $d(a,b)$ between two vertices $a$ and $b$ is
the length of a shortest path between them. We say that an edge $ab$ is \textit{horizontal to a vertex $u$} if $d(u,a) = d(u,b)$.
A cycle of length $k$ is denoted by $C_k$.
For a vertex $u$ of $G$, by $N(u)$ and $N[u]$ we denote the \emph{open} and the \emph{closed neighborhood} of $u$, respectively; hence $N[u]=N(u)\cup \{u\}$.
By extension, we define the open neighborhood $N(S)$ of a set $S$ to be
$(\cup_{a \in S}N(a)) \setminus S$.
By $N_i(u)$ we denote the set of vertices that are at distance $i$ from $u$.
Hence $N_0(u)=\{u\}$, $N_1(u)=N(u)$, etc.
The \emph{degree} of $u$ in $G$ is denoted by $d(u)$ and we always
denote the \emph{number of vertices} of $G$ by $n$, i.e. $n=|V(G)|$.
A vertex is \textit{dominating} if it is adjacent to all other vertices of a graph.
A \emph{non-edge} in a graph $G$ is a pair of non-adjacent vertices.  
A \textit{block} is a maximal subgraph without a cut-vertex.
Note that a block is either a $2$-connected subgraph, a single edge or an isolated vertex,
and every graph has a unique decomposition into blocks.
If $B\subseteq V(G)$, then $G[B]$ denotes the subgraph of $G$ induced by $B$.

For an unordered pair of vertices $a$ and $b$, an edge $e=uv$ is \textit{good}
for $\{a,b\}$ if $d(a,u) < d(a,v)$ and $d(b,v) < d(b,u)$.
Let $g(a,b)$ be the \emph{number of good edges} for $\{a,b\}$.
Note that every edge on a shortest path between $a$ and $b$ is a good edge for $\{a,b\}$, which gives $g(a,b) \geq d(a,b)$.

\begin{figure}[ht]
$$
\includegraphics{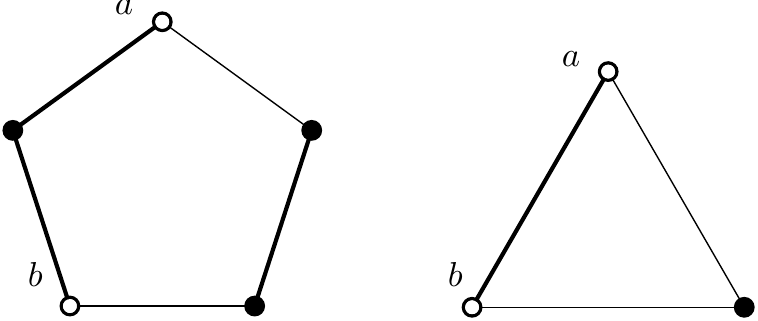}
$$
\caption{Cycles $C_5$ and $C_3$ with good edges for the pair $a$ and $b$
depicted in bold.}
\label{fig:c5}
\end{figure}

The concept of good edges has been introduced by Simi\'{c} et al. in~\cite{SimGutBal00} and used for an alternative definition of the Szeged index.
Observe that an edge $uv$ is good for exactly $n_{uv}(u)\cdot n_{uv}(v)$ pairs of vertices.
Therefore,
\begin{equation}
	\label{eq:good}	
	\Sz(G)=\sum_{uv\in E}n_{uv}(u)\cdot n_{uv}(v)=\sum_{\{a,b\}\subseteq V}g(a,b).
\end{equation}
Consequently, we obtain the following statement.

\begin{proposition} 
Let $G$ be a graph.
Then
$$
\eta(G) = \sum_{\{a,b\} \subseteq V} \big (g(a,b)-d(a,b) \big ).
$$
\end{proposition}

To simplify the notation, we write
$$
	\eta(a,b)=g(a,b)-d(a,b).
$$
Note that $g(a,a)=d(a,a)=0$. Since $g(a,b) \ge d(a,b)$ for every pair $a,b$, we easily obtain the known fact that $\Sz(G) \ge W(G)$.

Next, we define the \textit{contribution} $c_G(a)$ of a vertex $a$ in a graph $G$ as 
\begin{equation}
	\label{eq:c}
	c_G(a)=\sum_{b \in V}\eta(a,b)
	=\sum_{b \in V}\big(g(a,b)-d(a,b)\big).
\end{equation}
When there is no ambiguity from the context, we write $c(a)$ instead of $c_G(a)$. Therefore,
\begin{equation}
  \label{eq:eta_c}
  \eta(G) = \frac 12 {\sum_{a \in V} c(a)}.
\end{equation}

Now we are ready to prove Corollary~{\ref{th:main2}} using Theorem~{\ref{th:main1}}.

\begin{proof}[Proof of Corollary~\ref{th:main2}]
	Note first that for any $a,b$ from $B_i$ it holds $g_{B_i}(a,b) = g_G(a,b)$ and $d_{B_i}(a,b) = d_G(a,b)$.
	By Theorem~{\ref{th:main1}}, we have
	$$
		\eta(B_i) = \sum_{\{a,b\}\subseteq V(B_i)} \eta(a,b) =
		\sum_{\{a,b\}\subseteq V(B_i)}\big(g(a,b)-d(a,b)\big)\ge 2n_i-6,
	$$
	for every $i$, $1\le i\le k$. Since $g(a,b) \ge d(a,b)$ for every $\{a,b\}\subseteq V$, we conclude that
	$$
		\eta(G) = \sum_{\{a,b\}\subseteq V}\big(g(a,b)-d(a,b)\big) \ge \sum_{i=1}^k(2n_i-6).
	$$
\end{proof}


\section{Auxiliary results}
\label{sec:lemmas}

In this section we prove four technical lemmas which will be used
in the next sections to prove our theorems.

We first show that some vertices allow for an inductive argument.

\begin{lemma}
	\label{lem:induction}
	Let $G$ be a 2-connected graph. If $G$ contains a vertex $u$ such that $G-u$ is 2-connected, not complete, 
	and there exists a vertex $v$ such that $N[u]\subseteq N[v]$, then $\eta(G)-\eta(G-u)\geq 2$. 
\end{lemma}

\begin{proof}
	Let $G=(V,E)$ be a graph satisfying the assumptions, and let $G^*= G-u$.
	Since $G^*$ is not complete, neither is the graph $G$.
	Since $N[u] \subseteq N[v]$, we have that $d_{G^*}(a,b)=d_G(a,b)$ for any two vertices distinct from $u$. Hence,
	$$
		W(G)=W(G^*)+ \sum_{a\in V\setminus\{u\}} d(u,a).
	$$

	Let $q$ be the number of couples $(\{a,b\},w)$ such that $a, b$ are vertices of $G^*$ and $w$ is a neighbor of $u$ such that the edge $uw$ is good for $\{a,b\}$.
	Similarly as above, one can easily verify that
	$$
		\Sz(G) = \Sz(G^*)+ \sum_{a\in V\setminus\{u\}} g(u,a)+ q\,.
	$$ 

	Therefore,
	\begin{equation}
		\label{eq:twins}
		\eta(G)- \eta(G^*) = \sum_{a\in V\setminus\{u\}}\big(g(u,a)-d(u,a)\big) + q,
	\end{equation}
	and it suffices to prove that the right-hand side of (\ref{eq:twins}) is at	least 2. In order to do so,
	we consider the following contributions:

	\begin{itemize}
		\item[$(C1)$] \textit{$q$ is at least twice the number of non-edges in $N(u)$}.
				Indeed, for any two non-adjacent neighbors $w_1,w_2$ of $u$, the edges
				$uw_1$ and $uw_2$ are both good for $\{w_1,w_2\}$, and so each of the
				couples $(\{w_1,w_2\},w_1)$ and $(\{w_1,w_2\},w_2)$ contributes $1$ to $q$.
		\item[$(C2)$] \textit{For every vertex $a \not\in N[u]$, we have $g(u,a)-d(u,a) \geq 2 (p_a-1)$}, 
				where $p_a$ is the number of neighbors $w$ of $a$ with $d(u,w) < d(u,a)$. This follows from the fact that
				both edges, $uw$ and $aw$, are good for the pair $\set{u,a}$, for every $w$.
		\item[$(C3)$] \textit{For every edge $ab$ horizontal to $u$, and for every neighbor $x$ of $a$, 
				with $d(u,x) < d(u,a)$, such that $x$ is not adjacent to $b$, the edge $ax$ is good for $\set{u,b}$, i.e. it
				contributes an extra $1$ to $g(u,b)-d(u,b)$}.
	\end{itemize}	
	
	If $u$ is dominating, then, since $G$ is not a complete graph, there is a missing edge in $N(u)$ and $q \geq 2$ by $(C1)$.
	
	Assume now that $u$ is not dominating. We show that $\sum_{a\in V\setminus\{u\}} \big(g(u,a)-d(u,a)\big) \geq 2$.
	If there is exactly one vertex that does not belong to $N[u]$, then it must have at least two neighbors in $N(u)$ since $G$ is $2$-connected,
	and the conclusion follows from $(C2)$. 
	
	From now on, we assume that there are at least two vertices that do not belong to $N[u]$.
	Since $G$ is $2$-connected, there are at least two vertices in $N(N[u])$.
	If both have at least two neighbors in $N(u)$, we have $\sum_{a\in V\setminus\{u\}} \big(g(u,a)-d(u,a)\big)\geq 4$, by $(C2)$.
	So we may assume there is a vertex $a \in N(N[u])$ that has exactly one neighbor $x$ in $N(u)$.
	Since $G$ is $2$-connected, the graph $G-x$ is connected. 
	Let $P$ be a shortest path between $a$ and $u$ in $G-x$.
	Note that $P$ contains:
	\begin{itemize}
		\item{} a vertex $b$ having two neighbors $y_1$ and $y_2$ with $d(u,y_1)=d(u,y_2)<d(u,b)$; or 
		\item{} an edge $b_1b_2$ horizontal to $u$ with the property that for the other neighbors, $z_1$ and $z_2$, of $b_1$ and $b_2$ in $P$, respectively, 
				it holds $d(u,z_i)<d(u,b_i)$, $i \in \set{1,2}$, and $b_1z_2,b_2z_1 \notin E$. 
	\end{itemize}
	By $(C2)$ and $(C3)$, we infer that $\sum_{a\in V\setminus\{u\}} \big(g(u,a)-d(u,a)\big) \geq 2$ in either case, and so		
	we conclude that $\eta(G)-\eta(G^*)\geq 2$.	
\end{proof}

In the forthcoming lemma, we characterize the graphs $G$ for which the difference $\eta(G) - \eta(G-u)$ attains $2$.

\begin{lemma}
	\label{lem:K2n}
	Let $G$ be a $2$-connected graph of order $n$ containing a vertex $u$ such that $G-u$ is 2-connected, 
	not complete, and there exists a vertex $v$ such that $N[u]\subseteq N[v]$. 
	If $\eta(G)-\eta(G-u)=2$ and $G-u$ is isomorphic to $K_{n-1}^2$ or $K_{n-1}^{n-3}$, then $G$ is isomorphic to $K_n^2$ or $K_n^{n-2}$.
\end{lemma}

\begin{proof}
	Note that $n \geq 5$, since $G-u$ is $2$-connected and non-complete. 
	Let $G^* = G-u$ and let $\alpha$ be the number of ordered triples $(x,y,z)$ such that $y$ is adjacent to both $x$ and $z$ and $d(u,x) < d(u,y)$,
	and either 
	\begin{itemize}
		\item[$(a)$] $d(u,z)< d(u,y)$; or 
		\item[$(b)$] the edge $yz$ is horizontal to $u$ and $x$ is not adjacent to $z$.
	\end{itemize}

	From the proof of Lemma~\ref{lem:induction}, we know that either $u$ is dominating and, by $(C1)$, there is precisely one edge missing in $N(u)$,
	or $u$ is not dominating and, by $(C2)$ or $(C3)$, $\alpha \leq 2$.
	In the first case we immediately obtain that $G^*$ is isomorphic to $K^{n-3}_{n-1}$, and so $G$ is isomorphic to $K_n^{n-2}$.
	Therefore, we may assume that $u$ is not dominating, $N(u)$ induces a clique and $\alpha \leq 2$. 

	In what follows, we show that $|N(N[u])|=1$. Since $G^*$ is isomorphic to $K_{n-1}^{n-3}$ or $K_{n-1}^{2}$, it contains at least two dominating vertices.
	Note that every dominating vertex of $G^*$ is in $N(u)$ or $N(N[u])$.
	
	Suppose first there are two dominating vertices of $G^*$ in $N(u)$, say $w_1$ and $w_2$. 
	Then, for every $a\in N(N[u])$ both $(w_1,a,w_2)$ and $(w_2,a,w_1)$ contribute $1$ to $\alpha$ and hence $|N(N[u])|\le 1$.
	
	Suppose now there is a dominating vertex $w$ in $N(N[u])$. Then, due to the $2$-connectivity of $G$, $w$ has at least two
	neighbors, say $a_1$ and $a_2$, in $N(u)$ and consequently both $(a_1,w,a_2)$ and $(a_2,w,a_1)$ contribute $1$ to $\alpha$.
	If there is a vertex $z\ne w$ with $z\in N(N[u])$, then $wz\in E(G^*)$ as $w$ is dominating in $G^*$, 
	and either $a_i z \notin E(G^*)$ in which case $(a_i,w,z)$ contributes $1$ to $\alpha$, 
	or $a_1z,a_2z\in E(G^*)$ in which case $(a_1,z,a_2)$ contributes $1$ to $\alpha$.
	Consequently, $|N(N[u])|\le 1$ also in this case.
	
	So, in both cases we have $|N(N[u])|\le 1$ and since $u$ is not dominating in $G$, we conclude $|N(N[u])|=1$.
	Denote by $z$ the unique vertex of $N(N[u])$.
	Since $G$ is $2$-connected and $\alpha\le 2$, $z$ has exactly two neighbors
	in $N(u)$, which forms a clique.
	Consequently, $G^*$ is isomorphic to $K^2_{n-1}$ and $G$ is isomorphic to
	$K^2_n$.
\end{proof}

In the next lemma we also assume existence of two vertices, one of which dominates the other.

\begin{lemma}
	\label{lem:nopartialtruetwins}
	Let $G$ be a 2-connected non-complete graph distinct from $K_4^2$.
	Suppose that for every pair of vertices $u,v$ with $N[u]\subseteq N[v]$, $G-u$ 
	is complete or not $2$-connected, and at least one such a pair exists. Then for every
	two vertices $u,v$, where $N[u]\subseteq N[v]$, it holds that $G-u$ is not $2$-connected 
	and no block of $G-u$ is isomorphic to $K_t$, $K^2_t$ or $K^{t-2}_t$ for any $t$.
\end{lemma}

\begin{proof}
	Suppose there is a pair $u,v$ with $N[u]\subseteq N[v]$ and $G-u$ is complete.
	Then $u$ is not dominating in $G$ and $V = \set{u} \cup N(u) \cup N(N[u])$.
	If $N(N[u])$ contains at least two vertices $w_1,w_2$, then these vertices
	satisfy $N[w_1]=N[w_2]$, and $G - w_1$ is $2$-connected but not complete, which contradicts assumptions of the lemma.	
	So, let $w$ be the unique vertex in $N(N[u])$. Let $v_1,v_2$ be two vertices in $N(u)$
	(such vertices exist since $G$ is $2$-connected).
	We have $N[v_1]=N[v_2]$. If $N(u)$ is of size at least three, then $G - v_1$ is $2$-connected but not complete, a contradiction.
	If $N(u)$ is of size two, then $G$ is isomorphic to $K_4^2$, a contradiction.	
	
	By the above paragraph, we may assume that for every pair $u,v$ where $N[u]\subseteq N[v]$, the graph $G-u$ is not $2$-connected.
	Note that $v$ is contained in every block of $G-u$ and is the only cut-vertex in $G-u$.	
	
	Next, suppose for a contradiction that there is a pair $u,v$ with $N[u]\subseteq N[v]$ such that a block $C$ of $G-u$ is
	isomorphic to $K_t$, $K^2_t$ or $K^{t-2}_t$ for some $t$. Observe that each of $K_t$, $K^2_t$ and $K^{t-2}_t$ 
	has at least two dominating vertices since $t\ge 2$. Moreover, since $G$ is $2$-connected, there is a vertex $x \in N(u)$ distinct from $v$ in $C$.
	We distinguish two cases:
	\medskip

	{\bf Case~1.}
	{\it The vertex $v$ dominates $C$.}
	Then, every vertex $w\in V(C)\setminus\{v\}$ satisfies $N[w]\subseteq N[v]$,
	and so $G-w$ is not $2$-connected. If $|V(C)|=2$, i.e. if $C$ is a trivial block consisting of a single
	edge, then $x$ is the only vertex of $C$ distinct from $v$ and $d(x)=2$ since $xv,xu\in E$. 
	Since $G-x$ is not $2$-connected, while $G$ is, $G-x$ must be a single edge.
	Consequently $G$ is isomorphic to $K_3$, a contradiction. Hence $|V(C)|\ge 3$.
	
	Let $z\in V(C)\setminus \set{v,x}$. If $C$ contains a dominating vertex $w$ such that $w \notin \{v,z\}$, 
	then $G-z$ is $2$-connected, a contradiction. However since $C$ is isomorphic to $K_t$, $K^2_t$ or $K^{t-2}_t$, 
	it contains at least one dominating vertex $w$ other than $v$. This means that $C$ contains at most three vertices, and hence $|V(C)|=3$.
	Since $C$ has two dominating vertices, $C$ is isomorphic to $K_3$, and consequently $G-z$ is $2$-connected, a contradiction.
	\medskip

	{\bf Case~2.}
	{\it The vertex $v$ does not dominate $C$.}
	Then $C$ is not $K_t$.
	As mentioned above, there are at least two dominating vertices in $C$, and
	if $C$ is distinct from $K^2_t$, then there are at least three dominating
	vertices in $C$ since $C\ne K_4^2$.

	First suppose that $C$ has at least three dominating vertices $w_1,w_2,w_3$. 
	If $w_1u\notin E$, then $N[w_1]\subseteq N[w_2]$ and $G-w_1$ is
	$2$-connected, a contradiction. On the other hand if $w_1u\in E$, then $N[w_2]\subseteq N[w_1]$ and
	$G-w_2$ is $2$-connected, a contradiction.

	Thus, $C$ has exactly two dominating vertices, say $w_1$ and $w_2$. Then $C$ is isomorphic to $K^2_t$.	
	Observe that $|V(C)|\ge 4$. If $w_1u\in E$, then for every vertex $z$ of $C$, $z\notin\{v,w_1,w_2\}$, we
	have $N[z]\subseteq N[w_1]$ and $G-z$ is $2$-connected, a contradiction.
	Therefore $w_1u,w_2u\notin E$. Let $z$ be the vertex of degree $2$ in $C$. 
	Then $z \notin \set{w_1,w_2}$, and since $N[u] \subseteq N[v]$, also $z \ne v$.	
	If $zu \notin E$, then $N[z]\subseteq N[w_1]$ and $G-z$ is $2$-connected, a contradiction.
	On the other hand if $zu\in E$, then $N[w_1]\subseteq N[w_2]$ and $G-w_1$ is $2$-connected, a contradiction.
	
	\medskip
	This completes the proof of the lemma.
\end{proof}

We now describe how to deal with the case when there are no two vertices $u,v$ with $N[u]\subseteq N[v]$.

\begin{lemma}
	\label{lem:disconnectivity}
	Let $G$ be a 2-connected non-complete graph distinct from $C_5$ with no two vertices $u,v$ satisfying $N[u]\subseteq N[v]$.
	Then $c(u)\geq 4$ for every vertex $u$, and hence $\eta(G)\geq 2n$.
\end{lemma}

\begin{proof}
	Let $G$ be a graph satisfying the assumptions of the lemma. Assume for a contradiction 
	that there is a vertex $a$ such that $c(a) \leq 3$. Consider a breadth-first-search tree of $G$ rooted at $a$, with the
	levels $N_0,\ldots,N_p$, where $p$ is the eccentricity of $a$. Then $N_i=N_i(a)$.

	Observe first that $p\ge 2$. Indeed, if $p=1$, then for every $x\in V \setminus \set{a}$, we have
	$N[x]\subseteq N[a]$. Therefore, $G$ would contain only one vertex, namely the vertex $a$,
	which contradicts the 2-connectivity of $G$. 	

	Note also that for every $i$, $1\le i\le p-1$, we have $|E(N_i,N_{i+1})|\ge |N_{i+1}|$, and since $G$ is 2-connected, we have also
	$|N_i|\ge 2$. In what follows, for every $x\in V\setminus\{a\}$, we denote by $f(x)$
	the number such that $x \in N_{f(x)}$. When the vertex $x$ is clear from the context, we sometimes refer
	to the level $N_{f(x)-1}$ as the \textit{previous level}. 

	We say that a vertex $x$ has the property $P_1$ if it is adjacent to two vertices in the level $N_{f(x)-1}$. 
	Moreover, if there is a pair of adjacent vertices $x_1,x_2$ in a same level each having precisely one neighbor $y_1,y_2$, respectively, in the previous level,
	and $y_1 \ne y_2$, then both these vertices have the property $P_2$.

	In what follows we prove several claims about $G$.

	\begin{claim}
		\label{cl:P1}
		There is at most one vertex x with the property $P_1$.
		In the case when such a vertex $x$ exists, it has exactly two neighbors in $N_{f(x)-1}$. 
	\end{claim}

	\begin{proofclaim}
		Let $x$ be a vertex distinct from $a$, and let $y_1,\ldots,y_{\ell}$ be the neighbors of $x$ in $N_{f(x)-1}$.
		Suppose that $\ell \geq 2$ (this implies that $f(x) \ge 2$).
		For each $i \in \{1,\ldots, \ell \}$, let $z_i$ be a neighbor of $y_i$ in $N_{f(x)-2}$.
		Obviously, $z_i$'s need not be distinct.
		Anyway, similarly as $(C2)$ in the proof of Lemma~\ref{lem:induction}, 
		$$
			\eta(a,x) = g(a,x)-d(a,x) \geq 2\cdot (\ell-1),
		$$ 
		since all the edges $xy_i,y_iz_i$ are good for the pair $\{a,x\}$
		and there is a shortest path between $a$ and $x$ containing at most two of these edges. 
		Since $c(a)\le 3$, no vertex of $G$ has three neighbors in the previous level,
		and at most one has two neighbors in the previous level. This establishes the claim.
	\end{proofclaim}

	\begin{claim}
		\label{cl:P2}
		There is at most one pair of vertices $\{x,y\}$ with the property $P_2$.
		If such a pair exists, then there is no vertex $z\in V\setminus\{x,y\}$ with the property $P_1$.
	\end{claim}

	\begin{proofclaim}
		Let $x$ and $y$ be a pair of adjacent vertices with the property $P_2$. Then, there exist two distinct vertices 
		$w,z$ in $N_{f(x)-1}$ such that $x$ is adjacent to $w$ and $y$ is adjacent to $z$.
		We make no assumption about whether $x$ is adjacent to $z$, but by
		Claim~\ref{cl:P1}, we may assume that $y$ is not adjacent to $w$.
		Therefore the edge $wx$ is good for $\{a,y\}$.
		If $x$ is not adjacent to $z$, then the edge $yz$ is good for $\{a,x\}$, and
		if $x$ is adjacent to $z$, then analogously as in Claim~\ref{cl:P1},
		there are two edges around $z$, which are good for $\{a,x\}$.
		Anyway, $g(a,x)-d(a,x)\ge 1$ and $g(a,y)-d(a,y)\ge 1$.

		Now assume that we have two such pairs $\{x_1,y_1\}$ and $\{x_2,y_2\}$.
		If the two pairs are disjoint, then we have $\eta(a,u)\geq 1$ for each
		$u \in \{x_1,x_2,y_1,y_2\}$, which contradicts $c(a)\leq 3$.
		If they are not disjoint, then without loss of generality we may assume $x_1=x_2$,
		in which case $y_1\ne y_2$.
		From the previous analysis, we have $\eta(a,y_1)\geq 1$ and $\eta(a,y_2)\geq 1$.
		We claim that $g(a,x_1)-d(a,x_1)\geq 2$.
		Indeed, either $x_1$ has two neighbors in the previous level, or $x_1$ has
		exactly one neighbor in the previous level and for $z_1$ and $z_2$ being
		respective neighbors of $y_1$ and $y_2$ in the previous level, both $y_1z_1$
		and $y_2z_2$ are good for $\set{a,x_1}$. Hence the first conclusion holds.

		Finally, if there is a vertex $z\in V\setminus\{x,y\}$ with two neighbors in the
		previous level, then the conclusion follows as $\eta(a,z)\ge 2$, and hence $c(a) \ge 4$.
		This establishes the claim.
	\end{proofclaim}

	\begin{claim}
		\label{cl:level_p}
		In every connected component of $G[N_p]$ there is either a vertex with the property $P_1$ 
		or a pair of adjacent vertices which have the property $P_2$.
	\end{claim}

	\begin{proofclaim}
		Let $C$ be a connected component of $G[N_p]$.
		Suppose first that no vertex in $C$ has $P_1$.
		Then, since $G$ is $2$-connected, every vertex of $C$ has a neighbor in $N_p$.
		If there exists an edge $e=uv$ such that $u,v$ are not both adjacent to the
		same unique neighbor of $N_{p-1}$, then $u$ and $v$ have property $P_2$ and
		we are done.
		Hence, the endvertices of each edge $e$ of $C$ are both adjacent
		to a unique vertex $x_e\in N_{p-1}$.
		But, since $C$ is connected, all the $x_e$'s coincide and thus a
		unique vertex $x\in N_{p-1}$ is adjacent to all the vertices in
		$C$.  Since $p \ge 2$, $x$ is a cut-vertex of $G$ and so $G$ is
		not $2$-connected, a contradiction.

		Now, suppose that there is a vertex $x^*$ with $P_1$ and a pair
		of adjacent vertices $x_1,x_2$ with $P_2$ in $C$.
		By Claim~{\ref{cl:P2}}, $x^*\in\{x_1,x_2\}$.
		Without loss of generality we may assume that $x^*=x_1$.
		Denote by $y_1$ and $y_2$ the two distinct neighbors of $x_1$ in $N_{p-1}$,
		and denote by $y_3$ the unique neighbor of $x_2$ in $N_{p-1}$.
		If $y_3\notin\{y_1,y_2\}$, then $x_1y_1,x_1y_2$ are good for
		$\{a,x_2\}$ and $x_1y_1,x_1y_2,x_2y_3$ are good for $\{a,x_1\}$, 
		which gives $g(a,x_2)-d(a,x_2)\ge 2$ and $g(a,x_1)-d(a,x_1)\ge 2$
		and consequently $c(a)\ge 4$, a contradiction.
		Hence, $y_3\in\{y_1,y_2\}$.
		By Claim~{\ref{cl:P1}}, every neighbor of $x_2$ in $N_p$, other than $x_1$,
		is adjacent to a unique vertex of $N_{p-1}$.
		By Claim~{\ref{cl:P2}}, this vertex must be $y_3$, meaning that
		$N[x_2]\subseteq N[y_3]$, which contradicts the assumption.
	\end{proofclaim}

	By the previous claims, the vertices with the properties $P_1$ or $P_2$ are only in $N_p$.
	Moreover, there is either one vertex with $P_1$ or a pair of adjacent vertices with $P_2$
	and $N_p$ has only one component. Next claim deals with vertices which have neither $P_1$ nor $P_2$.

	\begin{claim}
		\label{cl:noPs}
		Let $u$ be a vertex which has neither $P_1$ nor $P_2$.
		Then $u$ has a neighbor in $N_{f(u)+1}$.
	\end{claim}

	\begin{proofclaim}
		By way of contradiction, assume that $u$ has no neighbors in $N_{f(u)+1}$.
		Then $u\ne a$.
		Since $u$ does not have $P_1$, it has a unique neighbor, say $v$, in $N_{f(u)-1}$.
		Since $G$ is 2-connected, $u$ must have a neighbor in $N_{f(u)}$.
		Let $w$ be a neighbor of $u$ in $N_{f(u)}$.
		If $vw\notin E$, then $u$ and $w$ have $P_2$, a contradiction.
		Hence, $vw\in E$.
		Consequently, $N[u]\subseteq N[v]$ contradicting the assumption.
	\end{proofclaim}

	By Claims~{\ref{cl:level_p}} and~\ref{cl:noPs}, every vertex in $N_p$ has either $P_1$ or $P_2$. 
	We consider the two cases separately.
	\medskip

	{\bf Case 1.}
	{\it There is $x\in N_p$ with the property $P_1$.} 
	By Claims~\ref{cl:P1} and~\ref{cl:P2}, $x$ is the only vertex in $G$ with the properties $P_1$ or $P_2$.
	Hence, by Claims~{\ref{cl:level_p}} and~\ref{cl:noPs}, we have $N_p=\{x\}$.
	Denote by $y_1$ and $y_2$ the two neighbors of $x$ in $N_{p-1}$.
	If there is $y^*\in N_{p-1} \setminus \set{y_1,y_2}$, 
	then it has neither $P_1$ nor $P_2$ by Claims~{\ref{cl:P1}} and~\ref{cl:P2},
	and so by Claim~{\ref{cl:noPs}}, $y^*$ has a neighbor in $N_p=\{x\}$,
	a contradiction.
	Thus, $N_{p-1}=\{y_1,y_2\}$.
	Let $z_i$ be a neighbor of $y_i$ in $N_{p-2}$, $1\le i\le 2$.

	First, suppose that $z_1\ne z_2$.
	Then $p\ge 3$.
	Let $b_i$ be the unique neighbor of $z_i$ in $N_{p-3}$, where $1\le i\le 2$.
	Then all the edges $b_1z_1,z_1y_1,y_1x,b_2z_2,z_2y_2,y_2x$ are good
	for $\{a,x\}$ which gives $\eta(a,x)\ge 3$.
	If $y_1y_2\in E$, then we have $N[x]\subseteq N[y_1]$, which
	contradicts the assumption; thus we have $y_1y_2\notin E$, which means that
	$y_1x$ is good for $\{a,y_2\}$.
	Hence $g(a,y_2)-d(a,y_2)\ge 1$ and consequently $c(a)\ge 4$, a contradiction.

	Thus, suppose that $z_1=z_2$.
	If $|V|>4$, then $z_1$ is a cut-vertex which contradicts the
	$2$-connectivity of $G$.
	Hence, $|V|=4$.
	Then $G$ is either $C_4$ or $C_4$ with a chord, depending on whether
	$y_1y_2$ is or is not in $E$.
	However, if $y_1y_2\in E$ then $N[y_1]\subseteq N[y_2]$, while if $G$ is $C_4$
	then $c(a)=4$ for every $a\in V(C_4)$; both these cases contradicting the
	assumptions.
	\medskip

	{\bf Case 2.}
	{\it There is a pair of adjacent vertices $x_1,x_2\in N_p$ with the property $P_2$.}
	Similarly as above, $N_p=\{x_1,x_2\}$.
	By Claim~{\ref{cl:level_p}}, both $x_1$ and $x_2$ have a unique neighbor in $N_{p-1}$.
	Let $y_i$ be the neighbor of $x_i$ in $N_{p-1}$, where $1\le i\le 2$.
	If $y_1 = y_2$, then $y_1$ is a cut-vertex which contradicts the $2$-connectivity of $G$. Hence $y_1 \neq y_2$.
	By Claim~{\ref{cl:P2}}, each $y_i$ has a unique neighbor $z_i$ in $N_{p-2}$.
	Analogously as in Case~1, we have $N_{p-1}=\{y_1,y_2\}$.

	Suppose that $z_1\ne z_2$.
	Then the edges $z_1y_1,y_1x_1$ are good for $\{a,x_2\}$ and 
	$z_2y_2,y_2x_2$ are good for $\{a,x_1\}$ which gives $\eta(a,x_2)\ge 2$
	and $\eta(a,x_1)\ge 2$ and consequently $c(a)\ge 4$, a contradiction.

	Thus, suppose that $z_1=z_2$.
	If $|V|>5$ then $z_1$ is a cut-vertex which contradicts the
	$2$-connectivity of $G$. Hence, $|V|=5$.
	Then $G$ is either $C_5$ or $C_5$ with a chord, depending on whether
	$y_1y_2$ is or is not in $E$. Since in the case $y_1y_2\in E$ we have $N[z_1]\subseteq N[y_1]$, $G$ is $C_5$.
\end{proof}


\section{Proof of Theorem~\ref{th:main3}}
	\label{sec:main}

We start with a lemma about $\eta(K^2_n)$ and $\eta(K^{n-2}_n)$.

\begin{lemma}
	\label{lem:eta(KK)}
	We have $\eta(K^2_n)=\eta(K^{n-2}_n)=2n-6$.
\end{lemma}

\begin{proof}
	Let $G\in\{K^2_n,K^{n-2}_n\}$. Denote by $u$ a vertex of smallest degree in $G$.
	Since the eccentricity of $u$ is $2$, we have $V(G)=N_0(u)\cup N_1(u)\cup N_2(u)$.
	Since both $G[N_0(u)\cup N_1(u)]$ and $G[N_1(u)\cup N_2(u)]$ are cliques,
	$\eta(x_1,x_2)>0$ for $x_1,x_2\in V(G)$ only if $u\in\{x_1,x_2\}$, say $u=x_1$,
	and the other vertex $x_2$ is in $N_2(u)$.
	In the case $G=K^2_n$, there are $n-3$ vertices $x_2$ in $N_2(u)$ and for
	each of them $\eta(u,x_2)=2$. On the other hand if $G=K^{n-2}_n$, there is a unique vertex $x_2$ in
	$N_2(u)$ and $\eta(u,x_2)=2(n-3)$.
\end{proof}

Now we prove Theorem~{\ref{th:main3}}.
By Lemma~{\ref{lem:eta(KK)}}, $\eta(K^2_n)=\eta(K^{n-2}_n)=2n-6$, and so
Theorem~{\ref{th:main1}} is a consequence of Theorem~{\ref{th:main3}}.

\begin{proof}[Proof of Theorem~\ref{th:main3}.]
	Let $G$ be a minimal counter-example to Theorem~{\ref{th:main3}}.
	Then $G$ is a $2$-connected non-complete graph on $n$ vertices, $n\ge 4$, 
	that is not isomorphic to $K^2_n$ or $K^{n-2}_n$, and Theorem~\ref{th:main3} holds for any smaller graph
	satisfying the assumptions of the theorem. We note that $\eta(C_5)=5$, thus $G$ is not isomorphic to $C_5$.

	By combining Lemma~\ref{lem:induction} and Lemma~\ref{lem:K2n}, we infer that for any two vertices $u,v$ in $G$ such that $N[u]\subseteq N[v]$, 
	the graph $G-u$ is not $2$-connected or is complete.
	From Lemma~\ref{lem:disconnectivity}, we have that there is a pair of vertices $u,v$ in $G$ such that $N[u]\subseteq N[v]$.
	Lemma~\ref{lem:nopartialtruetwins} guarantees that $G-u$ is not $2$-connected ($v$ being the only cut-vertex in $G-u$) 
	and no block of $G-u$ is complete or isomorphic to $K^2_t$ or $K^{t-2}_t$ for any $t \in \mathbb{N}$. Let $C_1,\ldots,C_k$ be the blocks of $G-u$. 
	Note that $k\geq 2$ and recall that $v$ belongs to each of the blocks as $N[u]\subseteq N[v]$ and $G$ is $2$-connected.
	Let $G_i=G[V(C_i)\cup\{u\}]$ for every $i$, $1 \le i \le k$. Note that each $G_i$ is a proper subgraph of $G$, 
	it is $2$-connected, and neither complete nor isomorphic to $K^2_t$ or $K^{t-2}_t$ for any $t$ (otherwise $C_i$ is complete, or isomorphic to $K_{t-1}^{t-3}$ or $K_{t-1}^2$, a contradiction).
	By minimality of $G$, we have $\eta(G_i) \geq 2|V(C_i)|-5$ for every $i$.

	We claim that 
	$$
		\eta(G) \geq \sum_{i=1}^k \eta(G_i)+ \sum_{\{a,b\}}\big(g(a,b)-d(a,b)\big),
	$$ 
	where the second sum runs over vertices $a,b$ from distinct blocks of $G-u$.
	This follows from the fact that distances for vertices in $G_i$ are left unchanged in $G$.
	Note that since $u$ and $v$ are counted in each $G_i$, it holds that 
	$$
	\sum_{i=1}^k \eta(G_i) \geq \sum_{i=1}^k \big(2|V(G_i)|-5\big) = 2n+2 \cdot 2 \cdot (k-1)-5k=2n-k-4.
	$$ 
	Since $G$ is $2$-connected, the vertex $u$ has a neighbor $w_i$
	(distinct from $v$) in each block $C_i$.
	Since $N[u]\subseteq N[v]$, we have
	$g(w_i,w_j)-d(w_i,w_j) = 2$ for every $i \neq j$.
	Therefore 
	$$
	\sum_{\{a,b\}}\big(g(a,b)-d(a,b)\big) \geq 2\, \binom{k}{2}, 
	$$
	where again the sum runs over vertices $a,b$ from distinct blocks.
	In total, to obtain $\eta(G) \geq 2n-5$, all we need is
	$2 \binom{k}{2} - k-4 \geq -5$, i.e. $\binom{k}{2} \geq \frac{k-1}2$.
	This is always true as $k \geq 2$.
\end{proof}


\section{Additional results}
\label{sec:add}

In this section we continue our analysis in order to prove several additional results regarding the difference between the Szeged and
the Wiener index. Recall that $c_G(a)$ is a contribution of $a$ to $\eta(G)$, see \eqref{eq:c} and \eqref{eq:eta_c}.

\begin{lemma}
	\label{th:blockdecomposition}
	Let $G$ be a connected graph with the blocks $B_1,\ldots,B_k$.
	Let $p_1,\ldots,p_k$ be integers such that for every $i$,
	with $1 \leq i \leq k$, it holds that $c_{G[B_i]}(u)\geq p_i$ for every vertex $u \in V(B_i)$. 
	Then $c_G(u)\geq \sum_{1 \leq i \leq k} p_i$ for every vertex $u \in V$. In particular, 
	$$
		\eta(G) \geq \frac{n}{2} \sum_{i = 1}^k p_i.
	$$
\end{lemma}

\begin{proof}
	Let $u$ be a vertex of $G$. Without loss of generality, assume that $u\in V(B_1)$.
	Then $\sum_{v\in V(B_1)\setminus\{u\}} \eta(u,v)= \sum_{v\in V(B_1)\setminus\{u\}} \big( g(u,v)-d(u,v) \big) \ge p_1$.
	Now choose a block $B_i$, $2\le i\le k$. 
	Let $v\in V(B_i)$ and let $w$ be a vertex of $B_i$ which is at the shortest distance from $u$. Then $w$ is a cut-vertex of $G$. 
	Therefore, every edge $xy$ being good for $\{w,v\}$ is also good for $\{u,v\}$,
	which implies $\sum_{v\in V(B_i)\setminus\{w\}} \eta(u,v)\ge \sum_{v\in V(B_i)\setminus\{w\}} \eta(w,v) \ge p_i$.
	Consequently, we infer $c_G(u)\ge\sum_{1\le i\le k} p_i$.
\end{proof}

Lemmas~\ref{lem:disconnectivity} and~\ref{th:blockdecomposition} together imply the following theorem.

\begin{theorem}
	\label{cor:five}
	Let $G$ be a connected graph with at least two blocks, neither of which is isomorphic to $C_5$ or a complete graph.
	Moreover, let $G$ do not contain two vertices $u,v$ with $N[u]\subseteq N[v]$. Then 
	$$
		\eta(G) \geq 4n.
	$$
\end{theorem}

The following two theorems have been conjectured by the AutoGraphiX computer program~\cite{Han10}, and  
confirmed in~\cite{CheLiLiu14} and~\cite{CheLiLiu14,CheLiLiuGut12}, respectively.

\begin{theorem}[Chen et al., 2014]
	\label{con:girth}
	Let $G$ be a connected graph with $n \ge 5$ vertices, at least one odd cycle, and girth at least 5.
	Then
	$$
		\eta(G) \ge 2n - 5.
	$$
	Moreover, the equality holds if and only if the graph is composed of $C_5$ and 
	a tree rooted at a vertex of $C_5$ or two trees rooted at two adjacent vertices of $C_5$.
\end{theorem}

\begin{figure}[ht]
	$$
		\includegraphics{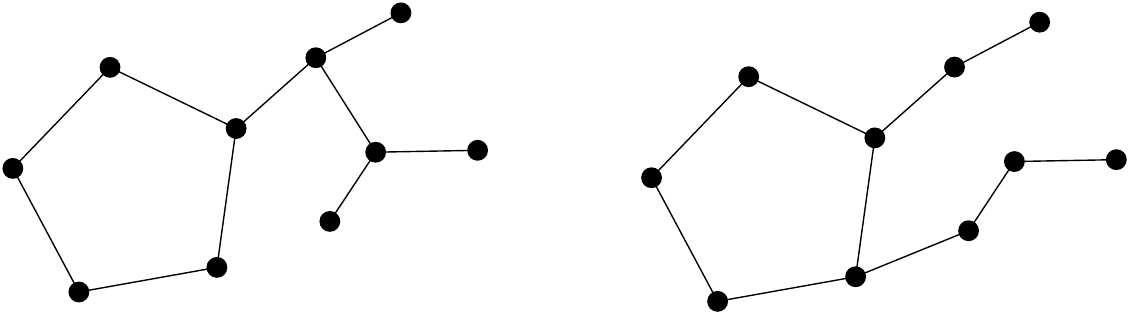}
	$$
	\caption{Example of two graphs for which the difference between their Szeged and Wiener index attains $2n-5$.}
	\label{fig:tightGirth4}
\end{figure}

\begin{theorem}[Chen et al., 2012, Chen et al., 2014]
	\label{con:bip}
	Let $G$ be a connected bipartite graph with $n \ge 4$ vertices and $m \ge n$ edges.
	Then
	$$
		\eta(G) \ge 4n - 8.
	$$
	Moreover, the equality holds if and only if the graph is composed of $C_4$ and a tree rooted at a vertex of $C_4$.
\end{theorem}
Additionally, Klav\v{z}ar and Nadjafi-Arani~\cite{KlaNad14b} extended Theorems~\ref{con:girth} and~\ref{con:bip} in terms of the girth and
the longest isometric cycle, respectively. 

Using our approach, we are able to prove stronger versions of Theorems~\ref{con:girth} and~\ref{con:bip}.
First, we strengthen Theorem~\ref{con:girth}. Notice that the next statement requires girth $4$ instead of $5$.

\begin{theorem}
	\label{th:girth}
	Let $G$ be a connected graph which has a triangle-free $2$-connected block distinct from $C_5$, 
	or at least two blocks isomorphic to $C_5$. Then 
	$$
		\eta(G) \geq 2n.
	$$
\end{theorem}

\begin{proof}
	Let $B$ be a $2$-connected block other than $C_5$ with girth at least $4$.
	Then $B$ is non-complete and it does not contain two vertices $u,v$ with $N[u]\subseteq N[v]$.
	By Lemma~{\ref{lem:disconnectivity}}, we have $c_B(u)\ge 4$ for every vertex $u\in V(B)$.
	On the other hand, $c_{C_5}(u)\ge 2$ for every vertex $u$ of $C_5$.
	Consequently by Lemma~{\ref{th:blockdecomposition}}, $c_G(u)\ge 4$ for
	every vertex $u$ in a graph $G$ satisfying the assumptions of theorem,
	which means that $\eta(G)\ge 2n$.
\end{proof}

As a corollary, we obtain an alternative proof Theorem~\ref{con:girth}.

\begin{proof}[Alternative proof of Theorem~\ref{con:girth}.]
	By Theorem~\ref{th:girth}, we have
	$$
		\eta(G) \geq 2n > 2n - 5,
	$$
	for every connected graph $G$ with girth at least $4$ and at least one $2$-connected block distinct from $C_5$, or
	at least two blocks isomorphic to $C_5$. So, we may assume that $G$ is a connected (unicyclic) graph with precisely one 
	$2$-connected block isomorphic to $C_5$.
	
	Let $C = abcde$ be the $C_5$ in $G$, and let $T_x$, $x \in \set{a,b,c,d,e}$, be the component of $G \setminus E(C)$
	containing $x$. Denote also $t_x = |V(T_x)|$. Obviously, $g(u,v)-d(u,v) = 0$ for every pair $u,v \in V(T_x)$.	
	Observe also that for any pair of adjacent vertices $x,y \in V(C)$, it holds $g(u,v)-d(u,v) = 0$ for every $u \in V(T_x)$ and $v \in V(T_y)$. 
	On the other hand, if $d(x,y) = 2$, then $g(u,v)-d(u,v) = 1$ for every $u \in V(T_x)$ and $v \in V(T_y)$.
	Since $t_a+t_b+t_c+t_d+t_e=n$, we have
	\begin{align}
		\label{eq:girth}
		\eta(G) & = \Sz(G) - W(G) = t_a t_c + t_b t_d + t_c t_e + t_d t_a + t_e t_b \nonumber \\
				& = n^2-\frac{(t_a+t_b)^2}2-\frac{(t_b+t_c)^2}2 -\frac{(t_c+t_d)^2}2-\frac{(t_d+t_e)^2}2-\frac{(t_e+t_a)^2}2.
	\end{align}
	We leave to the reader to verify that the minimum value of~\eqref{eq:girth} is achieved 
	in the case when at most two trees $T_x$ and $T_y$, $x,y \in \set{a,b,c,d,e}$, have more than one vertex and 
	$x$ is adjacent to $y$. This completes the proof.
\end{proof}

Now we consider bipartite graphs.

\begin{theorem}
	\label{th:ifbipartite}
	If $G$ is a $2$-connected bipartite graph, then either $G$ is isomorphic to $C_4$ 
	or $c(u) \geq 8$ for every vertex $u \in V$. In particular, if $G$ is not isomorphic to $C_4$, then 
	$$
		\eta(G) \geq 4n.
	$$
\end{theorem}

\begin{proof}
	To prove the theorem, we will use analogous approach as in the proof of Lemma~\ref{lem:disconnectivity}. 
	Suppose, to the contrary, that $a$ is a vertex of $G$ with $c(a) \le 7$.
	Consider a breadth-first-search tree of $G$ rooted at $a$, with the levels $N_0,\dots,N_p$.

	Observe first that since $G$ is bipartite, there are no two vertices
	$u,v$ with $N[u]\subseteq N[v]$, no two vertices in a common level $N_i$
	are adjacent, and $p \ge 2$. Let $x\in N_p$. Then $x$ has no neighbor in $N_p$.
	Let $y_1,\dots,y_\ell$ be the neighbors of $x$ in the previous level.
	Since $G$ is $2$-connected, $\ell\ge 2$.
	Let $z_i$ be a neighbor of $y_i$ in the level $f(x)-2$.
	Note that $z_i$ might be equal to $z_j$ for some $i\neq j$.
	Anyway, the edges $z_iy_i$ and $y_ix$ are good for $\{a,x\}$, which gives
	$g(a,x) - d(a,x) \ge 2(\ell-1) \ge 2$.
	Additionally, every edge $xy_j$ is good for $\{a,y_i\}$, if $i \neq j$, since $y_iy_j\notin E$.
	Thus, $g(a,y_i) - d(a,y_i) \ge \ell - 1 \ge 1$ for every $i$.
	The contribution of $a$ is thus at least $4$, and at least $10$ if $\ell \ge 3$.
	It follows that $x$ is the only vertex with two neighbors in the previous
	level, and moreover, it has precisely two such neighbors. Consequently $N_p=\{x\}$.

	Consider now the levels $N_0,\dots,N_{p-1}$.
	Since $G[N_i]$ is a graph without edges and $x$ is the only vertex with at least
	two neighbors in the previous level, $G$ is isomorphic to a cycle.
	It easy to calculate that in the case when $p \ge 3$, the contribution of $a$ is at least $9$.
\end{proof}

From Theorem~\ref{th:ifbipartite} we also infer:

\begin{theorem}
	\label{cor:bip}
	Let $G$ be a connected graph which has a bipartite $2$-connected block distinct from $C_4$, 
	or at least two blocks isomorphic to $C_4$. Then 
	$$
		\eta(G) \geq 4n.
	$$ 	
\end{theorem}

\begin{proof}
	If $B$ is a block inducing a bipartite subgraph other than $C_4$, then
	$c_B(u)\ge 8$ for every vertex $u\in V(B)$, by Theorem~{\ref{th:ifbipartite}}.
	On the other hand, $c_{C_4}(u)\ge 4$ for every vertex $u$ of $C_4$.
	Consequently by Lemma~{\ref{th:blockdecomposition}}, $c_G(u)\ge 8$ for
	every vertex $u$ in a graph $G$ satisfying the assumptions of corollary,
	which means that $\eta(G)\ge 4n$.
\end{proof}

Similarly as above, we can derive Theorem~{\ref{con:bip}} from Theorem~{\ref{cor:bip}}.

\begin{proof}[Alternative proof of Theorem~\ref{con:bip}.]
	By Theorem~\ref{cor:bip}, we have
	$$
		\eta(G) \geq 4n > 4n - 8,
	$$
	for every bipartite connected graph $G$ with at least one $2$-connected block distinct from $C_4$, or
	at least two blocks isomorphic to $C_4$. So, we may assume that $G$ is a bipartite connected (unicyclic) graph with precisely one 
	$2$-connected block isomorphic to $C_4$.

	Let $C = abcd$ be the $C_4$ in $G$, and let $T_x$, $x \in \set{a,b,c,d}$, be the component of $G \setminus E(C)$	
	containing $x$. Denote also $t_x = |V(T_x)|$. Again, $g(u,v)-d(u,v) = 0$ for every pair $u,v \in V(T_x)$.	
	For any pair of adjacent vertices $x,y \in V(C)$, it holds $g(u,v)-d(u,v) = 1$ for every $u \in V(T_x)$ and $v \in V(T_y)$.
	If $d(x,y) = 2$, then $g(u,v)-d(u,v) = 2$ for every $u \in V(T_x)$ and $v \in V(T_y)$. Thus,
	\begin{align}
		\label{eq:bip}
		\eta(G) & = \Sz(G) - W(G) = t_a t_b + t_b t_c + t_c t_d + t_d t_a + 2 (t_a t_c + t_b t_d) \nonumber \\
				& =\frac{n^2-(t_a-t_c)^2-(t_b-t_d)^2}2.
	\end{align}
	It is straightforward to verify that the minimum value of~\eqref{eq:bip} is achieved in the case when at most one tree $T_x$, $x \in \set{a,b,c,d}$, has 
	more than one vertex. This completes the proof.
\end{proof}


\section{Revised Szeged index}
	\label{sec:revised}

In this section we consider a variation of the Szeged index.
The \textit{revised Szeged index} is defined as
$$
	\Sz^*(G) = \sum_{uv \in E}\Big(n_{uv}(u)+\frac{n_0(u,v)}2\Big)\Big(n_{uv}(v)+\frac{n_0(u,v)}2\Big),
$$ 
where $n_0(u,v) = n - n_{uv}(u) - n_{uv}(v)$.
Similarly as above, we define
$$
	\eta^*(G) = \Sz^*(G) - W(G).
$$

For every vertex $a \in V$, we set $h(a)$ to be the number of edges horizontal to $a$. Note that
$$
	\sum_{a \in V} h(a) = \sum_{uv \in E} n_0(u,v).
$$
Clearly, in every bipartite graph we have $\eta^*(G)=\eta(G)$. We therefore focus on non-bipartite graphs exclusively. 
Note that in a non-bipartite graph, it holds that $h(a) \geq 1$ for every vertex $a$.

A conjecture about the difference between the revised Szeged index and the Wiener index of a graph 
has also been proposed by AutoGraphiX~\cite{Han10} and confirmed by Chen et al.~\cite{CheLiLiu14}.

\begin{theorem}[Chen et al., 2014]
	\label{conj:revised}
	Let $G$ be a non-bipartite connected graph. Then 
	$$
		\eta^*(G) \geq \frac{n^2}{4}+n-\frac{3}{2}.
	$$
	Moreover, the equality is attained when the graph is composed of $C_3$ and a tree rooted at a vertex of $C_3$.
\end{theorem}

Recently, Klav\v{z}ar and Nadjafi-Arani~\cite{KlaNad14} extended the result in a way that it takes the number of edges into account. 
Here we present a similar result.

\begin{lemma}
	\label{th:revised}
	Let $G$ be a non-bipartite connected graph. Then
	$$
		\eta^*(G) \geq \eta(G)+\frac{n^2}{4}+\frac{n}2.
	$$
	More precisely, we have 
	$$
		\eta^*(G) \geq \eta(G)+\frac{n^2}{4}+\frac{n}2+\frac{n+2}{4}\cdot \sum_{a\in V} \big(h(a)-1 \big).
	$$
\end{lemma}

\begin{proof}
	We have
	\begin{align*}
		\eta^*(G) = & \ \Sz^*(G) - W(G) \\
		= & \sum_{uv \in E}\Big(n_{uv}(u)+\frac{n_0(u,v)}2\Big)\Big(n_{uv}(v)+\frac{n_0(u,v)}2\Big)-\sum_{\{u, v\} \subseteq V}d(u,v)\\
		= & \sum_{uv \in E}\Big(n_{uv}(u)n_{uv}(v)+\frac{n_0(u,v)}2\big(n_{uv}(v)+n_{uv}(u)+n_0(u,v)\big)-\frac{n_0(u,v)^2}4\Big)-\sum_{\{u, v\} \subseteq V}d(u,v)\\
		= & \sum_{uv \in E}\big(n_{uv}(u)n_{uv}(v)\big)-\sum_{\{u, v\} \subseteq V}d(u,v)+\sum_{uv \in E}\Big(\frac{n_0(u,v)\cdot n}2-\frac{n_0(u,v)^2}4 \Big)\\
		= & \ \eta(G)+\frac12\sum_{uv \in E}n_0(u,v)\Big(n-\frac{n_0(u,v)}2\Big).
	\end{align*}
	Note that $n_0(u,v)\leq n-2$ for any edge $uv$, and so
	\begin{align*}
		\eta^*(G) \geq & \ \eta(G)+\frac12\sum_{uv \in E} n_0(u,v) \Big(n-\frac{n-2}2 \Big)\\ 
		= & \ \eta(G)+\frac{n+2}4\sum_{uv \in E}n_0(u,v)\\ 
		= & \ \eta(G)+\frac{n+2}4\sum_{a \in V}h(a).
	\end{align*}		
	Recall that $h(a) \geq 1$ for any vertex $a$. Thus,
	\begin{align*}
		\eta^*(G) \ge & \ \eta(G)+\frac{(n+2)\cdot n}4+ \frac{n+2}4\sum_{a \in V}\big(h(a)-1\big).
	\end{align*}
\end{proof}

We use Lemma~\ref{th:revised} to prove the following stronger version of Theorem~\ref{conj:revised}.

\begin{theorem}
	\label{cor:treetriangle}
	Let $G$ be a non-bipartite connected graph.
	Then
	$$
		\eta^*(G) \geq \frac{n^2}{4} + n,
	$$
	unless $G$ is obtained from a tree by expanding a single vertex into $C_3$.
\end{theorem}

\begin{proof}
	Suppose, to the contrary, that $G$ is a non-bipartite graph with $\eta^*(G) < \frac{n^2}{4} + n$.
	By Lemma~\ref{th:revised}, there is neither a vertex $a$ with
	$h(a) \geq 3$, nor two vertices $a,b$ with $h(a), h(b) \geq 2$. 
	In particular, $G$ does not contain a diamond (a $K_4$ without one edge) as
	a subgraph, and hence $G$ has no clique larger than a triangle.
	Similarly, the graph does not contain two edge-disjoint odd cycles.
	Therefore, there is at most one triangle in $G$.

	Additionally, it holds that $\eta(G) < \frac{n}2$.
	By Lemma~\ref{th:blockdecomposition}, no block satisfies  $c(a) \geq 1$
	for every vertex.
	Observe that every vertex $u$ of $C_5$ satisfies $c(u)=1$.
	Thus by Lemma~{\ref{lem:disconnectivity}}, $G$ does not contain
	$2$-connected block without a triangle.

	Hence, $G$ has a unique $2$-connected block $B$ and this block contains a
	triangle.
	Suppose that there is a vertex $a$ in $B$ satisfying $c_B(a) = 0$.
	Moreover, suppose that $a$ is not a dominating vertex in $B$.
	Then the breadth-first-search tree for $B$ rooted in $a$ contains at least three levels. 
	Since $B$ is $2$-connected, there must be in the last level either a vertex
	with two neighbors in the previous level or an edge whose endpoints are
	adjacent to two distinct neighbors in the previous level.
	Both these cases contradict to $c(a)=0$. Hence, $a$ is dominating in $B$.
	Since there is at most one triangle in $B$ and $B$ is $2$-connected, we conclude that $B$ is isomorphic to $C_3$.
\end{proof}

Finding $\eta^*(G)$ for all graphs $G$ obtained from a tree by expanding a single vertex into $C_3$ we obtain an alternative proof of Theorem~\ref{conj:revised}.

\begin{proof}[Alternative proof of Theorem~\ref{conj:revised}.]
	By Theorem~\ref{cor:treetriangle}, we have 
	$$
		\eta^*(G) \ge \frac{n^2}{4} + n > \frac{n^2}{4} + n - \frac{3}{2}
	$$
	for every connected non-bipartite graph $G$, unless $G$ is obtained from a tree by expanding a single vertex into $C_3$.
	So, we may assume that $G$ is a graph obtained from a tree by expanding one vertex into $C_3$.
	
	Let $C = abc$ be the $C_3$ in $G$ and let $T_x$, $x \in \set{a,b,c}$, be the component of $G \setminus E(C)$
	containing $x$. Denote also $t_x = |V(T_x)|$.
	Then
	\begin{align*}
		\eta^*(G) & = \Sz^*(G) - W(G) = \Sz^*(G) - \Sz(G) + \Sz(G) - W(G) = \Sz^*(G) - \Sz(G).
	\end{align*}
	The last equality holds as $G$ is a block graph. 
	Note also that $n_0(u,v) = 0$, unless $uv \in \set{ab,bc,ac}$, in which cases it holds
	$n_0(a,b) = t_c$, $n_0(b,c) = t_a$, and $n_0(a,c) = t_b$.
	Since $n_x(x,y)=t_x$ for $x,y\in\{a,b,c\}$, $x\ne y$, we have
	\begin{align}
		\label{eq:rev}
		\eta^*(G) & =  \sum_{uv \in E} \Bigg( \ \Bigg (n_u(uv) + \frac{n_0(u,v)}{2} \Bigg) \Bigg(n_v(uv) + \frac{n_0(u,v)}{2}\Bigg) - n_u(uv)n_v(uv) \ \Bigg) \nonumber \\
				  & =  \Big( \big(n_a(ab) + \tfrac{1}{2}t_c\big)\big(n_b(ab) + \tfrac{1}{2}t_c \big) - n_a(ab)n_b(ab) \Big) \nonumber \\
				  & +  \Big( \big(n_b(bc) + \tfrac{1}{2}t_a\big)\big(n_c(bc) + \tfrac{1}{2}t_a \big) - n_b(bc)n_c(bc) \Big) \nonumber \\
				  & +  \Big( \big(n_a(ac) + \tfrac{1}{2}t_b\big)\big(n_c(ac) + \tfrac{1}{2}t_b \big) - n_a(ac)n_c(ac) \Big) \nonumber \\
				  & = \tfrac{1}{4}(t_a^2 + t_b^2 + t_c^2) + t_a t_b+ t_a t_c + t_b t_c \nonumber \\
				  & = \tfrac 5{12}n^2 -\tfrac 1{12}\big( (t_a-t_b)^2+(t_b-t_c)^2+(t_c-t_a)^2 \big).
	\end{align}
	The minimum value of~\eqref{eq:rev} is achieved when at most one tree $T_x$, $x\in\{a,b,c\}$, has more than one vertex.
	This completes the proof.
\end{proof}

\bigskip
\noindent
{\bf Acknowledgment.} 
The second author acknowledges partial support by Slovak research grants VEGA 1/0007/14, VEGA 1/0065/13 and APVV 0136--12.
The research was partially supported by Slovenian research agency ARRS (program no. P1--0383 and project no. L1--4292), 
and by French agency CampusFrance (Proteus PHC project 31284QJ).

\bigskip


\bibliographystyle{plain}

\bibliography{mainBib}

\end{document}